\documentclass[reqno]{amsart}
 \usepackage{amsaddr}
\usepackage{amssymb}
\usepackage{leftidx}
\newtheorem{theorem}{Theorem}[section]
\newtheorem{lemma}[theorem]{Lemma}
\newtheorem{prop}[theorem]{Proposition}

\theoremstyle{definition}
\newtheorem{definition}[theorem]{Definition}

\newtheorem{cor}[theorem]{Corollary}

\theoremstyle{remark}
\newtheorem{remark}[theorem]{Remark}

\numberwithin{equation}{section}

\newcommand{\C}{\mathbb{C}}

\newcommand{\Z}{\mathbb{Z}}

\newcommand{\Span}{{\rm span}}
\newcommand{\Zhu}{{\rm Zhu}}
\newcommand{\Res}{{\rm Res}}
\newcommand{\End}{{\rm End}}

\newcommand{\Id}{{\rm Id}}
\newcommand{\vertex}{(V, Y, {\mathbf 1}, \omega)}

\def \bp{\begin{prop}\label}
\def \ep{\end{prop}}
\def \bt{\begin{theorem}\label}
\def \et{\end{theorem}}

\begin{document}

\title[Higher level Zhu algebras are subquotients]{Higher level Zhu algebras are subquotients of universal enveloping algebras}

\author{Xiao He}

\address{D\'epartement de math\'ematiques et de statistique,\\ Universit\'e Laval, Qu\'ebec, Canada}
\email{xiao.he.1@ulaval.ca}



\begin{abstract}
We prove that  higher level Zhu algebras of a vertex operator algebra are isomorphic to subquotients of its universal enveloping algebra. 
\end{abstract}

\maketitle

\section*{Introduction}
The Zhu algebra of a vertex operator algebra plays very important roles. For example, there is a one-to-one correspondence between the isomorphism classes of irreducible admissible modules of the vertex operator algebra and the isomorphism classes of irreducible modules of its Zhu algebra \cite{Zhu, Zhu96}. The Zhu algebra was first introduced by Y. C. Zhu in his thesis paper \cite{Zhu} in 1990. Then in 1998, Dong, Li and Mason \cite{DLM} generalized the Zhu algebra to a series of associative algebras which we call higher level Zhu algebras. Dong, Li and Mason proved similar results about the correspondence between representations of higher level Zhu algebras and those of vertex operator algebras. In  \cite{FZ}, Frenkel and Zhu mention that the Zhu algebra is isomorphic to a subquotient of the universal enveloping algebra. In this paper, we generalize this statement and prove that higher level Zhu algebras are isomorphic to subquotients of the universal enveloping algebra. 

The organization of this paper is as follows. In the first section, we give a review of vertex operator algebras, their modules, the Zhu algebra and higher level Zhu algebras. Then we recall the universal enveloping algebra of a vertex operator algebra in the second section. Section 3 is the core of this paper, where we establish the isomorphisms between higher level Zhu algebras and subquotients of the universal enveloping algebra.

\bigskip
\section{Vertex operator algebras, their modules, the Zhu algebra and higher level Zhu algebras.}
In this section, we briefly recall the definitions, mainly for fixing the notation. For details, we refer to the papers \cite{DLM, FHL, Zhu}.

All vector spaces and tensor products are considered over the complex numbers $\C$.

\bigskip 
\subsection{Vertex operator algebras and their modules} 
For a vector space $U$, the space of formal Laurent series in $z$ with coefficients in $U$ is defined to be
$$U[[z, z^{-1}]]:=\left\{\sum_{n\in \Z}u_nz^n~|~u_n\in U\right\}.$$
It contains the subspace of Laurent series,
$$U((z)):=\left\{\sum_{n\in \Z}u_nz^n~|~u_n\in U, u_n=0 \mbox{~for~} n\ll 0\right \}.$$
The formal differential and the formal residue of an element in $U[[z, z^{-1}]]$ are defined to be 
$$\dfrac{d}{dz}\left(\sum_{n\in \Z}u_nz^n\right):=\sum_{n\in \Z}nu_nz^{n-1},\qquad \Res_z\left(\sum_{n\in \Z}u_nz^n\right):=u_{-1}.$$

\begin{definition}
A \emph{vertex operator algebra} is a quadruple $\vertex$,  where $V=\bigoplus_{n\in \Z}V_n$ is a $\Z$-graded vector space with $\dim \, V_n<\infty$ for all $n\in \Z$, $V_n=0$ for $n\ll 0$, and $Y$ is a  linear map from $V$ to  $\End\,V[[z, z^{-1}]]$ sending $v\in V$ to $Y(v, z):=\sum_{n\in \Z}v_nz^{-n-1}$, which we call  the \emph{vertex operator} of $v$, that satisfies the following four axioms:
\begin{itemize}
\item [(1)] The \emph{vacuum element} ${\mathbf 1 }\in V$ satisfies: $Y({\mathbf 1}, z)=\Id_V$ and $v_{i}{\mathbf 1}=\delta_{i, -1}v$ for all $v\in V$ and $i\geq -1$.   
\item[(2)] $Y(u, z)v\in V((z)),  \mbox{~i.e., $u_nv=0$ whenever $n\gg 0$, for all~} u, v\in V.$ We call this the \emph{truncation condition}.
\item[(3)] $\omega$ is a special element in $V$, which we call a \emph{Virasoro element}. Let $L(n):=w_{n+1}$. Then 
$$[L(m), L(n)]=(m-n)L(m+n)+\delta_{m+n, 0}\dfrac{m^3-m}{12}\mathbf{c},$$
where $\mathbf{c} \in \C$ is called the \emph{rank} or $\emph{central charge}$ of $V$. Moreover, 
$$L(0)|_{V_n}=n\Id_{V_n}\quad \mbox{and} \quad Y(L(-1)u, z)=\dfrac{d}{dz}Y(u, z) \mbox{~for all~} u\in V.$$
\item[(4)] The \emph{Jacobi identity}: For $\ell, m, n\in \Z$ and $u, v\in V$,
$$\sum_{i\geq 0}(-1)^i{\ell \choose i}\left(u_{m+\ell-i}v_{n+i}-(-1)^{\ell}v_{n+\ell-i}u_{m+i}\right)=\sum_{i\geq 0}{m\choose i}(u_{\ell+i}v)_{m+n-i}.$$
\end{itemize} 
\end{definition}
\bigskip
From the Jacobi identity, we can get the following two important formulas.  By setting $\ell=0$ in the Jacobi identity, we have
\begin{align*}
 [u_m, v_n]=\sum_{i\geq 0}{m\choose i}(u_iv)_{m+n-i}.\qquad \mbox{\tag*{(Commutator formula)}}                      
\end{align*}
By setting $m=0$ in the Jacobi identity, we have
\begin{align*}
(u_\ell v)_n=\sum_{i\geq 0}(-1)^i{\ell\choose i}(u_{\ell-i}v_{n+i}-(-1)^\ell v_{\ell+n-i}u_i).\qquad \mbox{\tag*{(Iterate formula)}}
\end{align*}
An element $v\in V_n$ is called homogeneous of conformal weight $n$, and we denote by $\Delta_v=n$. Whenever we use the notation $\Delta_v$, we assume that $v$ is homogeneous. 

\begin{definition}
A \emph{weak module} for a vertex operator algebra $V$ is a vector space $M$, with a linear map $Y_M: V \rightarrow \End\, M[[z, z^{-1}]]$ sending $v$ to $Y_M(v, z)=\sum_{}v^M_nz^{-n-1}$, that satisfies:
\begin{itemize}
\item [(1)] $Y_M({\mathbf 1}, z)=\Id_M$ and $Y_M(v, z)w\in M((z))  \mbox{~for all~} v\in V, w\in M.$
\item[(2)] For $\ell, m, n\in \Z \mbox{~and~} u, v\in V$,
$$\sum_{i\geq 0}(-1)^i{\ell\choose i}\left(u^M_{m+\ell-i}v^M_{n+i}-(-1)^{\ell}v^M_{n+\ell-i}u^M_{m+i}\right)=\sum_{i\geq 0}{m\choose i}(u_{\ell+i}v)^M_{m+n-i}.$$
\end{itemize}
A weak module $M$ is called \emph{admissible} if it is also $\Z_{\geq 0}$-graded $M=\bigoplus_{n\geq 0}M_n$ and satisfies:
\begin{itemize}
\item[(3)] For any homogeneous element $v\in V$, we have 
$$v_n^MM_m\subseteq M_{m+\Delta_v-n-1}.$$ 
\end{itemize}
\end{definition}
\bigskip
Submodules, quotient modules, simple modules and semisimple modules can be defined in the obvious way. 

\bigskip 
\subsection{Zhu algebra}
Let $\vertex$ be a vertex operator algebra. Following \cite{Zhu}, 
we will construct an associative algebra $\Zhu(V)$ associated to $V$.

Let 
$$O(V):=\Span\{u\circ v~|~ u, v\in V\},$$
where the linear product $\circ$ is defined on homogeneous $u\in V$ by 
\begin{equation}
u\circ v:=\Res_{z}\left(Y(u,z)v\frac{(1+z)^{\Delta_ u}}{z^{2}}\right)=\sum_{i\geq 0}{\Delta_u\choose i}u_{i-2}v.
\end{equation}

Define a product $*$ on $V$ by the formula:
\begin{align}
u*v&:=\Res_z\left(Y(u,z)v\frac{(1+z)^{\Delta_u}}{z}\right)=\sum_{i\geq 0}{\Delta_u\choose i}u_{i-1}v.
\end{align}

The subspace $O(V)$ is known to be a two-sided ideal of $V$ under $*$ (cf. \cite{Zhu}).

Let
$$\Zhu(V):=V/O(V).$$

\begin{theorem}\cite{Zhu}.  The product $*$ induces an associative algebra structure on $\Zhu(V)$ with identity ${\bf 1}+O(V).$ The element $\omega+O(V)$ is central in $\Zhu(V)$. 
\end{theorem}


For an admissible $V$-module $M=\bigoplus_{n\geq 0}M_n$, we call $M_n$ the $n$-th level and $M_0$ the top level of $M$.  Denote by $o^M(u):=u^M_{ \Delta_u-1}$ for all homogeneous $u\in V$ and extend linearly to $V$. Then $o^M(u)M_n\subseteq M_n$. In particular, $o^M(u)$ preserves the top level.

\begin{lemma}\cite{Zhu}. For an admissible $V$-module $M$, when restricted to the top level $M_0$, the identities $$o^M(u)o^M(v)=o^M(u*v) \mbox{~~and~~} o^M(u')=0$$
hold for all $u, v\in V$ and $u'\in O(V)$. Thus, the top level $M_0$ is a $\Zhu(V)$-module under the action $(u+O(V))\cdot m=o^M(u)m$. 
\end{lemma}  

The correspondence $M\mapsto M_0$ gives a functor, which we denote by $\Omega_0$, from the category of admissible $V$-modules to the category of $\Zhu(V)$-modules. On the other hand, Zhu constructed another functor $L^0$ from the category of $\Zhu(V)$-modules to the category of admissible $V$-modules in his thesis paper \cite{Zhu}. Given a $\Zhu(V)$-module $U$ with action $\pi$, $L^0(U)$ is an admissible module for $V$ with top level being $U$. Moreover, we have $\pi(v)m=o^{L^0(U)}(v)m$ for all $m\in U$ and $v\in V.$

\begin{theorem}\cite{Zhu}.
The two functors $\Omega_0, L^0$ are mutually inverse to each other when restricted to the full subcategory of completely reducible admissible $V$-modules and the full subcategory of completely reducible $\Zhu(V)$-modules. 
\end{theorem}

\bigskip 
\subsection{Higher level Zhu algebras}
Let $\vertex$ be a vertex operator algebra. Following \cite{DLM}, we are going to construct an associative algebra $A_n(V)$ for each nonnegative integer $n$, which we will call the \emph{level $n$ Zhu algebra}\footnote{We follow the terminology as in the reference \cite{Jethro} for the twisted case.}, with $A_0(V)$ being exactly the Zhu algebra $\Zhu(V)$. We will call the algebras $A_n(V)$ \emph{higher level Zhu algebras} when $n\geq 1.$

Recall that $L(n)=\omega_{n+1}$, where $\omega$ is the Virasoro element of $V$. For $n\geq 0$, let 
$$O_n(V):=\Span\{u\circ_n v, L(-1)u+L(0)u~|~ u, v\in V\},$$
where the linear product $\circ_n$ is defined on homogeneous $u\in V$ by 
\begin{align}
u\circ_n v:&=\Res_{z}\left(Y(u,z)v\frac{(1+z)^{\Delta_ u+n}}{z^{2n+2}}\right) \nonumber \\
&=\sum_{i=0}^{\infty}{\Delta_u+n\choose i}u_{i-2n-2}v.
\end{align}

Define a product $*_n$ on $V$ by the formula:
\begin{align}
u*_nv&:=\sum_{m=0}^{n}(-1)^m{m+n\choose n}\Res_z\left(Y(u,z)v\frac{(1+z)^{\Delta_u+n}}{z^{n+m+1}}\right) \nonumber\\
&=\sum_{m=0}^n\sum_{i=0}^{\infty}(-1)^m{m+n\choose n}{\Delta_u+n\choose i}u_{i-m-n-1}v.  
\end{align}

The subspace $O_n(V)$ is a two-sided ideal of $V$ under $*_n$ (cf. \cite{DLM}).

Let
$$A_n(V):=V/O_{n}(V).$$

\begin{theorem}\cite{DLM}.  The product $*_n$ induces an associative algebra structure on $A_n(V)$ with identity ${\bf 1}+O_n(V).$ The identity map on $V$ induces a surjective algebra homomorphism from $A_n(V)$ to $A_{n-1}(V).$
\end{theorem}

\begin{remark}
Note that $L(-1)u+L(0)u=u \circ \mathbf{1}$ and $u\circ_0 v=u\circ v$, so $O_0(V)$ coincides with $O(V)$. Moreover, as $u*_0v=u*v$, the algebra  $A_0(V)=\Zhu(V)$ is just the Zhu algebra.
\end{remark}

We have an inverse system of associative algebras:
\begin{equation}\label{zhualgebrasystem}
A_0(V) \twoheadleftarrow A_{1}(V)\twoheadleftarrow \cdots \twoheadleftarrow A_n(V)\twoheadleftarrow A_{n+1}(V)\twoheadleftarrow \cdots.
\end{equation}

These higher level Zhu algebras play similar roles to that of the Zhu algebra in the representation theory of vertex operator algebras. To describe the relationship between the representations of $A_n(V)$ and those of $V$, we recall a Lie algebra associated to $V$. 

Consider the vector space $V\otimes \C[t, t^{-1}]$ and the linear operator $$\partial:=L(-1)\otimes \Id+\Id\otimes \dfrac{d}{dt}.$$ Let 
$$\hat{V}:=\dfrac{V\otimes \C[t, t^{-1}]}{\partial (V\otimes \C[t, t^{-1}])}.$$
Denote by $v(m)$ the image of $v\otimes t^m$ in $\hat{V}$ for $v\in V$ and $m\in \Z$. The vector space $\hat{V}$ is a $\Z$-graded Lie algebra by defining the degree of $v(m)$ to be $\Delta_v-m-1$ and the Lie bracket: 
\begin{equation}\label{Liebracket}
[u(m), v(n)]=\sum_{i\geq 0}{m\choose i}(u_iv)(m+n-i) \mbox{~for $u, v\in V$}.
\end{equation} 
As the Lie bracket (\ref{Liebracket}) in $\hat{V}$ is the just the commutator formula in $V$, the natural map from $\hat{V}$ to $\End\, V$ sending $v(m)$ to $v_m$ is a Lie algebra homomorphism. In this way, we can consider a $V$-module as a $\hat{V}$-module.

Denote the homogeneous subspace of $V$ of degree $m$ by $\hat{V}(m)$. Then $\hat{V}(0)$ is a Lie subalgebra of $\hat{V}$. Consider the Lie algebra structure of $A_n(V)$ with Lie bracket $[u, v]=u*_n v-v*_n u$ for $u, v\in V$. 

\begin{lemma}\cite{DLM}.
There is a surjective Lie algebra homomorphism from $\hat{V}(0)$ to $A_n(V)$ for each $n$, sending $o(v):=v(\Delta_v-1)$ to $v+O_n(V)$.
\end{lemma}

For an admissible module $M$ and  a nonnegative integer $n$, define the subspace 
\begin{equation}
\Omega_n(M):=\{m\in M~|~\hat{V}(-k)m=0 \mbox{~if~} k>n\}.
\end{equation} 

\begin{theorem}\cite{DLM}.
Let $M=\bigoplus_{i\geq 0}M_i$ be an admissible $V$-module.
\begin{itemize} 
\item[(i)] The subspace $\Omega_n(M)$ admits an $A_n(V)$-module structure under the action $v\cdot m=o^M(v)m$, with each $M_i$ being a submodule for $0\leq i\leq n$.
\item[(ii)]If $M$ is simple, then $\Omega_n(M)=\bigoplus_{0\leq i\leq n}M_i$ and each $M_i$ for $0\leq i\leq n$ is a simple $A_n(V)$-module. Moreover, $M_i$ and $M_j$ are inequivalent $A_n(V)$-modules when $i\neq j$.
\end{itemize}
\end{theorem}
Since there is a surjective homomorphism $A_n(V)\twoheadrightarrow A_{n-1}(V)$, the subspace $\Omega_{n-1}(M)\subset \Omega_n(M)$ is naturally an $A_n(V)$-module. Let
\begin{equation*}
\Omega_n/\Omega_{n-1}(M):=\dfrac{\Omega_n(M)}{\Omega_{n-1}(M)}.
\end{equation*}
Then $\Omega_n/\Omega_{n-1}$ defines a functor from the category of admissible $V$-modules to the category of $A_n(V)$-modules. The good thing is that this functor has an inverse when restricted to an appropriate subcategory. In \cite{DLM}, the authors constructed a functor $L^n$ from the category of $A_n(V)$-modules to the category of admissible $V$-modules, such that, for a given $A_n(V)$-module $U$ with action $\pi$,  $\Omega_n/\Omega_{n-1}(L^n(U))\cong U$ as $A_n(V)$-modules, i.e., $o^{L^n(U)}(v)m=\pi(v)m$ for all $v\in V$ and $m\in U$. 

\begin{theorem}\cite{DLM}.
The two functors $\Omega_n/\Omega_{n-1}$ and $L^n$ are inverse to each other when restricted to the full subcategory of completely reducible admissible $V$-modules and the full subcategory of completely reducible $A_n(V)$-modules which do not factor through $A_{n-1}(V)$. 
\end{theorem} 

\bigskip 
\section{The universal enveloping algebra  and its subquotients}

To define the universal enveloping algebra of a vertex operator algebra, we need to introduce a completion notation, as the Jacobi identity contains infinite sums. 

Let $U(\hat{V})$ be the universal enveloping algebra of the Lie algebra $\hat{V}$ that we constructed in Section 2. Then it inherits a natural $\Z$-grading from  $\hat{V}$, $U(\hat{V})=\bigoplus_{n\in \Z}U(\hat{V})_n$. The zero component $U(\hat{V})_0$ contains $U(\hat{V}(0))$, the universal enveloping algebra of $\hat{V}(0),$ as a subalgebra.  For $n\in \Z$ and $k\in \Z_{\leq 0}$, let 
$$U(\hat{V})^k_n=\sum_{i\leq k}U(\hat{V})_{n-i}U(\hat{V})_i \quad\mbox{~~and~~}\quad U(\hat{V}(0))^k=U(\hat{V}(0))\bigcap U(\hat{V})^k_0.$$
Then 
$$\cdots \subseteq U(\hat{V})_n^k\subseteq U(\hat{V})_n^{k+1}\subseteq \cdots \subseteq U(\hat{V})_n^0=U(\hat{V})_n$$
and 
$$\cdots \subseteq U(\hat{V}(0))^k\subseteq U(\hat{V}(0))^{k+1}\subseteq \cdots \subseteq U(\hat{V}(0))^0=U(\hat{V}(0))$$
are well-defined filtrations of $U(\hat{V})_n$ and $U(\hat{V}(0))$, respectively. Moreover, we have 
\begin{equation*}
\bigcap_{k}U(\hat{V})_n^k=0, \qquad \bigcup_{k}U(\hat{V})_n^k=U(\hat{V})_n.
\end{equation*}
Hence, the filtration $\{U(\hat{V})^k_n\}_{k\leq 0}$ forms a fundamental neighborhood system of $U(\hat{V})_n$. Let $\tilde{U}(\hat{V})_n$ be the completion of $U(\hat{V})_n$ with respect to this filtration, i.e., infinite sums are allowed in $\tilde{U}(\hat{V})_n$, and for any given $k$,  only finitely many terms are contained in $U(\hat{V})_n^k\setminus U(\hat{V})_n^{k+1}$. Let $\tilde{U}(\hat{V}(0))$ be the completion of $U(\hat{V}(0))$ with respect to the filtration $\{U(\hat{V}(0))^k\}_{k\leq 0}$, it is obviously a subspace of $\tilde{U}(\hat{V})_0$.

Let \begin{equation*}
\tilde{U}(\hat{V}):=\bigoplus_{n\in \Z}\tilde{U}(\hat{V})_n.
\end{equation*}
The space $\tilde{U}(\hat{V})$ becomes a $\Z$-graded topological ring with each component $\tilde{U}(\hat{V})_n$ being complete.  The subspace $U(\hat{V})$ is a dense subalgebra of $\tilde{U}(\hat{V})$ with $U(\hat{V})_n$ being dense in $\tilde{U}(\hat{V})_n$ for all $n$. The completion $\tilde{U}(\hat{V})$ is called a degreewise completed topological ring in the theory of quasi-finite algebras studied by A. Matsuo et al. in \cite{MNT}.

Denote by 
\begin{align*}
\langle \mbox{Vac}\rangle: \quad & \mathbf{1}(i)=\delta_{i, -1}, \mbox{~for all~} i\in \Z,\\
\langle\mbox{Vir}\rangle : \quad &[L(m), L(n)]=(m-n)L(m+n)+\delta_{m+n, 0}\dfrac{m^3-m}{12}\mathbf{c}, \mbox{~for all~} m ,n\in \Z,\\
J_{m, n, \ell}^{u, v}: \quad  &\sum_{i\geq 0}(-1)^i{\ell\choose i}\left(u(m+\ell-i)v(n+i)-(-1)^{\ell}v(n+\ell-i)u(m+i)\right)\\
&=\sum_{i\geq 0}{m\choose i}(u_{\ell+i}v)(m+n-i), \,\mbox{~for~} u, v\in V \mbox{~and~} m,n, \ell\in \Z.
\end{align*}
\begin{remark}
The element $L(n)$ should be considered as the image of $\omega\otimes t^{n+1}$ in $\hat{V}$. The Jacobi relation $J_{m, n, \ell}^{u, v}$ is now well-defined in $\tilde{U}(\hat{V})$.
\end{remark}  

\begin{definition}
The universal enveloping algebra  $U(V)$ of $V$ is the quotient of $\tilde{U}(\hat{V})$ by the relations: $\langle \mbox{Vac}\rangle, \langle\mbox{Vir}\rangle$ and $\langle J_{m, n, \ell}^{u, v}~|~ u, v\in V, m,n,\ell\in \Z \rangle$.
\end{definition}
\bigskip 
All the relations $\langle \mbox{Vac}\rangle, \langle\mbox{Vir}\rangle$ and $J_{m, n, \ell}^{u, v}$ are homogeneous, so the universal enveloping algebra $U(V)$ inherits a natural $\Z$-grading from $\tilde{U}(\hat{V})$. 

The image $U(V(0))$ of $\tilde{U}(\hat{V}(0))$ in $U(V)$ is obviously contained in  $U(V)_0$, which is a subalgebra of $U(V)$.

Let $$U(V)_0^k:=\sum_{i\leq k}U(V)_{-i}U(V)_i \quad \mbox{~and~} \quad U(V(0))^k:=U(V(0))\bigcap U(V)_0^k.$$ 
Then $\dfrac{U(V)_0}{U(V)_0^k}$ and $\dfrac{U(V(0))}{U(V(0))^k}$ inherit an associative algebra structure, as ${U(V)_0^k}$ and $U(V(0))^k$ are two-sided ideals of $U(V)_0$ and $U(V(0))$, respectively. By the obvious inclusions $U(V)_0^k\subseteq U(V)_0^{k+1}$ and $U(V(0))^k\subseteq U(V(0))^{k+1}$, we have two  inverse systems of algebras:
\begin{align}
&\frac{U(V)_0}{ U(V)_0^{-1}}\twoheadleftarrow \frac{U(V)_0}{ U(V)_0^{-2}}  \twoheadleftarrow \cdots \twoheadleftarrow   \frac{U(V)_0}{ U(V)_0^{-n}} \twoheadleftarrow \frac{U(V)_0}{U(V)_0^{-n-1}} \twoheadleftarrow \cdots,\\
&\frac{U(V(0))}{ U(V(0))^{-1}}\twoheadleftarrow \frac{U(V(0))}{ U(V(0))^{-2}} \twoheadleftarrow \cdots \twoheadleftarrow   \frac{U(V(0))}{ U(V(0))^{-n}} \twoheadleftarrow \frac{U(V(0))}{U(V(0))^{-n-1}}  \twoheadleftarrow \cdots .
\end{align}

Our goal is prove that these two inverse systems of associative algebras are both isomorphic to the inverse system given by higher level Zhu algebras (\ref{zhualgebrasystem}). More precisely, we are going to prove that 
$$A_n(V)\cong \frac{U(V)_0}{U(V)_0^{-n-1}}\cong\frac{U(V(0))}{ U(V(0))^{-n-1}}.$$

\bigskip
\section{The isomorphisms}
One of our motivations for this study is the paper \cite{FZ} of  I. Frenkel and Y. C. Zhu, where they mention a statement  that the Zhu algebra is isomorphic to a subquotient of the universal enveloping algebra. In this section, we prove that all higher level Zhu algebras are also isomorphic to subquotients of the universal enveloping algebra.

For simplicity, we use the following notation: For $u, v\in V \mbox{~and~} m, n, \ell\in \Z$, let
\begin{align*}
^{1}J_{m, n, \ell}^{u, v}:&=\sum_{i\geq 0}{m\choose i}(u_{\ell+i}v)(m+n-i),\\
^2J_{m, n, \ell}^{u, v}:&=\sum_{i\geq 0}(-1)^i{\ell\choose i}(u(m+\ell-i)v(n+i)-(-1)^{\ell}v(n+\ell-i)v(m+i)).
\end{align*}
They are just the two sides of the Jacobi identity $J_{m ,n ,\ell}^{u, v}$, so in the universal enveloping algebra, we have $^{1}J_{m, n, \ell}^{u, v}=\, ^2J_{m, n, \ell}^{u, v}$. 

We use the following notation, which is defined for homogeneous elements and extended linearly to the whole $V$.
$$J_n(u):=u(\Delta_u-1+n).$$
A good property for this notation is that the degree of $J_n(u)$ is always $-n$.

Let
\begin{align*}
^{(1)}J_{m, n, \ell}^{u, v}:&=\, ^{1}J_{m+\Delta_u-1, n+\Delta_v-1, \ell}^{u, v} \\
&= \sum_{i\geq 0}{m+\Delta_u-1\choose i}J_{m+n+\ell}(u_{\ell+i}v),\\
^{(2)}J_{m, n, \ell}^{u, v}:&=\, ^2J_{m+\Delta_u-1, n+\Delta_v-1, \ell}^{u, v}\\
&=\sum_{i\geq 0}(-1)^i{\ell\choose i}(J_{m+\ell-i}(u)J_{n+i}(v)
-(-1)^{\ell}J_{n+\ell-i}(v)J_{m+i}(u)).
\end{align*}
Every term in the expressions $^{(1)}J_{m, n, \ell}^{u, v}, \, ^{(2)}J_{m, n, \ell}^{u, v}$ is of the same degree $-m-n-\ell.$

The following lemma is very important in the proof of the main theorem. 
\begin{lemma}\label{surjectivelemma}
Every element $\sum J_{n_1}(u_1)\cdots J_{n_m}(u_m)$ in $\dfrac{U(V)_0}{U(V)_0^{-n}}$ can be expressed as $J_0(u(w))$ for some $u(w)\in V.$
\end{lemma}
\begin{proof}
We only need to prove the claim for  monomials $w=J_{n_1}(u_1)\cdots J_{n_m}(u_m)$. 
Define the degree of $w$ to be $m$, i.e., the number of factors of it. Then a degree one element in $U(V)_0$ is just an element of the form $J_0(u)$ for some $u\in V$, and we need to show that every monomial in the quotient $\dfrac{U(V)_0}{U(V)_0^{-n}}$ is congruent to a degree one element.

We use induction on $m$, the degree of the monomial $w$. If $m=1$, there is nothing we need to do. Let $m=k\geq 2$ and assume that for every monomial of degree less than $k$, it is congruent to a degree one element in the quotient $\dfrac{U(V)_0}{U(V)_0^{-n}}$.

Use the formula in Corollary \ref{keycoro} for $J_{n_{m-1}}(u_{m-1})J_{n_m}(u_m)$, where $$-s=n_{m-1},\, t=n_m, \, u=u_{m-1},\, v=u_m.$$ 
In Corollary \ref{keycoro}, choose $N$ big enough, so that $\min\{N+n_m, N\}>n$. Then $J_{k+n_m}(u_m)$ and $J_{N+1+i}(u_{m-1})$ are both contained in $\bigoplus_{j\leq -n}U(V)_{j}$ for $k\geq N+1$, and $w=J_{n_1}(u_1)\cdots J_{n_m}(u_m)$ is congruent to a linear combination of the following lower degree monomials:
$$J_{n_1}(u_1)\cdots J_{n_{m-2}}(u_{m-2})J_{n_m+n_{m-1}}((u_{m-1})_{-N+n_{m-1}-i-j-1}u_m).$$

By induction, these lower degree monomials are congruent to degree one monomials, so $w$ is itself congruent to a degree one monomial.  
\end{proof}

Now we are in a position to prove the isomorphisms between higher level Zhu algebras and subquotients of the universal enveloping algebra. 

\begin{theorem}
We have the isomorphism for $n\geq 0$, 
\begin{equation}\label{Isom}
A_n(V)\cong \dfrac{U(V)_0}{U(V)_0^{-n-1}}.
\end{equation}

\end{theorem}
\begin{proof}
Let $\varphi$ be the map from $V$ to $U(V)_0$ sending $v$ to $o(v)$, where $o(v)$ is the image of $v(\Delta_v-1)$ in $U(V)$ for homogeneous $v$ and extended linearly to $V$. Combine it with the canonical quotient map from  $U(V)_0$ to $\dfrac{U(V)_0}{U(V)_0^{-n-1}}$. Then Lemma \ref{surjectivelemma} tells us that this map is surjective.

First of all, we show that $\varphi$ factors through $A_n(V)$, i.e., $\varphi(O_n(V))\subset U(V)_0^{-n-1}$.

Recall that
$O_n(V)=\Span\{u\circ_n v, L(-1)u+L(0)u~|~ u, v\in V, u \mbox{~homogeneous}\},$
where 
\begin{align*}
u\circ_n v&=\sum_{i=0}^{\Delta_u+n}{\Delta_u+n\choose i}u_{i-2n-2}v.
\end{align*}
As $\varphi(L(-1)u+L(0)u)\equiv 0$, we only need to prove that 
$\varphi(u\circ v)\in U(V)_0^{-n-1}.$
Assume that $u, v$ are both homogeneous. Then $\Delta_{u_{i-2n-2}v}=\Delta_u+\Delta_v+2n+1+i,$ and 
\begin{align*}
\varphi(u\circ v)&=\sum_{i=0}^{\Delta_u+n}{\Delta_u+n\choose i}(u_{i-2n-2}v)(\Delta_u+\Delta_v+2n+i)\\
&=\, ^{(1)}J_{n+1, n+1, -2n-2}^{u, v} \\
&= \, ^{(2)}J_{n+1, n+1, -2n-2}^{u, v}\\
&=\sum_{i\geq 0}(-1)^i{-2n-2\choose i}J_{-n-1-i}(u)J_{n+1+i}(v)\\
&\qquad-\sum_{i\geq 0}(-1)^{i}{-2n-2\choose i}J_{-n-1-i}(v)J_{n+1+i}(u).
\end{align*}

As $\deg\,J_{n+1+i}(v)=\deg\,J_{n+1+i}(u)\leq -n-1,$ so $\varphi(u\circ v)\in U(V)_0^{-n-1}$.

Next we prove that $\varphi$ is an algebra homomorphism, i.e., $\varphi(u*_nv)=\varphi(u)\varphi(v)$.

Recall that 
\begin{align*}
u*_nv&=\sum_{m=0}^n\sum_{i=0}^{\infty}(-1)^m
{m+n\choose n}{\Delta_u+n\choose i}u_{i-m-n-1}v.
\end{align*}
So 
\begin{align*}
\varphi&(u*_nv)
\\
&=\sum_{m=0}^n\sum_{i=0}^{\infty}(-1)^m
{m+n\choose n}{\Delta_u+n\choose i}(u_{i-m-n-1}v)(\Delta_u+\Delta_v+m+n-i).
\end{align*}
By letting $s=t=0$ and $N=n$ in Corollary \ref{keycoro}, we have 
\begin{align*}
J_{0}(u)J_0(v)&\equiv \sum_{j=0}^{n}{-n-1\choose j}J^{(1)}_{n+1, j, -n-1-j}(u, v) \quad \mbox{~mod~} U(V)_0^{-n-1}\\
&\equiv\sum_{j=0}^{n}\sum_{i\geq 0}{-n-1\choose j}{n+\Delta_u\choose i}J_0(u_{-n-1+i-j}) \quad \mbox{~mod~} U(V)_0^{-n-1}\\
&\equiv \sum_{j=0}^n\sum_{i=0}^{\infty}(-1)^j
{j+n\choose j}{\Delta_u+n\choose i}J_0(u_{i-j-n-1}v) \quad \mbox{~mod~} U(V)_0^{-n-1}
\end{align*}
that is, $\varphi(u*_nv)=\varphi(u)\varphi(v)$.

Finally, we want to construct an inverse map for $\varphi$. Every element of $\dfrac{U(V)_0}{U(V)_0^{-n-1}}$ can be expressed as $J_0(u)+U(V)_0^{-n-1}$. We want to define the map $\varphi^{-1}$ from $\dfrac{U(V)_0}{U(V)_0^{-n-1}}$ to $A_n(V)$ sending $J_0(u)+U(V)_0^{-n-1}$ to $u+O_n(V)$. Once we prove that this is a well-defined map, it is an inverse for $\varphi.$ 
The well-definedness requires that whenever $J_0(u)\in U(V)_0^{-n-1}$, we have $u\in O_n(V)$. Then $\varphi^{-1}$ does not depend on the representatives of an element of $\dfrac{U(V)_0}{U(V)_0^{-n-1}}$ and really gives a map. Consider the functor $L^n$ constructed by Dong, Li and Mason in \cite{DLM}, and its action on the left regular module $A_n(V)$.  If $J_0(u)\in U(V)_0^{-n-1}$, $J_0(u)$ will kill the subspace $\Omega_n(L^n(A_n(V)))$ by definition. In particular, it will kill the quotient $\Omega_n/\Omega_{n-1}(L^n(A_n(V)))$, which is isomorphic to $A_n(V)$ as $A_n(V)$-modules. So $J_0(u)v=u*_nv$ for all $v\in V$, and in particular, for $v={\mathbf 1}$, which is the identity element of $A_n(V)$. So we have $u*_n{\mathbf 1}=J_0(u){\mathbf 1}=0$, which implies that $u\in O_n(V)$.
\end{proof}

\begin{cor}
The Zhu algebra is isomorphic to a subquotient of the universal enveloping algebra.  
$$\Zhu(V)=A_0(V)\cong \dfrac{U(V)_0}{U(V)_0^{-1}}.$$
\end{cor}

Recall that there is a surjective Lie algebra homomorphism from $\hat{V}(0)$ to $A_n(V)$, which induces a surjective associative algebra homomorphism from $U(\hat{V}(0))$ to $A_n(V)$. Combine it with the isomorphism (\ref{Isom}), we can conclude that $U(\hat{V}(0))$ is a dense subalgebra of $U(V)_0.$

\begin{cor}
The subalgebra $U(\hat{V}(0))$ is dense in $U(V)_0$, i.e., $U(\hat{V}(0))+U(V)_0^{-n}=U(V)_0 \mbox{~for all~} n\leq 0.$ So we have the isomorphisms:  $$ \frac{U(V)_0}{U(V)_0^{-n-1}}\cong\frac{U(V(0))}{ U(V(0))^{-n-1}},  \mbox{~for $n\geq 0$}.$$
\end{cor}

Let $C_{2}V:=\Span \{u_{-2}v~|~ u, v\in V \}$. A vertex operator algebra is called $C_2$-cofinite if $\dim \dfrac{V}{C_2V}<\infty.$ In  \cite{MNT}, the authors proved that if a vertex operator algebra is $C_2$-cofinite, then all the subquotients $\dfrac{U(V)_0}{U(V)_0^{-n-1}}$ are finite dimensional. With the isomorphisms between $A_n(V)$ and these subquotients, we easily get the corollary below. 

\begin{cor}
If $V$ is a $C_2$-cofinite vertex operator algebra, then all higher level Zhu algebras are finite dimensional. 
\end{cor}

\bigskip
\emph{Acknowledgements:} The author would like to express his thanks to Professor Atsushi Matsuo and Professor Yongchang Zhu for their kind communications. The author would also like to thank China Scholarship Council (File No. 201304910374) and l'Institut des sciences math\'ematiques for their financial support.

\bigskip 
\appendix
\section{}
In this appendix, we prove a combinatorial identity and one of its corollary, which we used in the proof of the main theorem.

For a negative integer $n$ and a positive integer $k$,  recall that 
\begin{equation}\label{negativen}
{n\choose k}=\dfrac{n(n-1)\cdots (n-k+1)}{k!}=(-1)^k{-n+k-1\choose k}.
\end{equation}

Recall the notation in Section 3.  The statement of the following lemma was suggested by Atsushi Matsuo.
\begin{lemma}\label{Key Lemma}
For any integers $s, t$ and $N$ satisfying $N+s\geq 0$, 
\begin{align*}
X &:=\sum_{j=0}^{N}{-N-s-1\choose j}\,^{(2)}J_{N+1, t+j, -N-s-1-j}^{u, v} \nonumber \\
&=J_{-s}(u)J_t(v)+\sum_{k\geq N+1}\sum_{j=0}^N(-1)^{j}{N+s+j\choose j}{N+s-k\choose k-j}J_{-k-s}(u)J_{k+t}(v)\\
&\quad\, -\sum_{j=0}^N\sum_{i\geq 0}(-1)^{N+s+1}{N+s+j\choose j}{N+s+j+i\choose i}J_{t-N-s-1-i}(v)J_{N+1+i}(u). 
\end{align*}
\end{lemma}
\begin{proof}
By definition, $^{(2)}J_{N+1, t+j, -N-s-1-j}^{u, v}=A-B$, where
\begin{align*}
A=&\sum_{i\geq 0}(-1)^i{-N-s-1-j\choose i}J_{-s-j-i}(u)J_{t+j+i}(v),\\
B=&\sum_{i\geq 0}(-1)^{-N-s-1-j+i}{-N-s-1-j\choose i}J_{t-N-s-1-i}(v)J_{N+1+i}(u).
\end{align*}
So $X=C-D$, where
\begin{align*}
C&=\sum_{j=0}^{N}{-N-s-1\choose j}A\\
 &=\sum_{j=0}^{N}
\sum_{i\geq 0}(-1)^j{N+s+j\choose j} {N+s+j+i\choose i}J_{-s-j-i}(u)J_{t+j+i}(v),\\
D&=\sum_{j=0}^{N}{-N-s-1\choose j}B\\
&=\sum_{j=0}^{N}\sum_{i\geq 0}(-1)^{N+s+1}{N+s+j\choose j}{N+s+i+j\choose i}J_{t-N-s-1-i}(v)J_{N+1+i}(u).
\end{align*}
We used the formula (\ref{negativen}) in the above calculation.

Let $k=i+j$ in the expression of $C$, then
\begin{align}
C &=\sum_{j=0}^{N}
\sum_{k\geq j}(-1)^{j}{N+s+j\choose j} {N+s+k\choose k-j}J_{-s-k}(u)J_{k+t}(v)\nonumber\\
&=\sum_{k=0}^N\sum_{j=0}^{k}(-1)^{j}{N+s+j\choose j} {N+s+k\choose k-j}J_{-s-k}(u)J_{k+t}(v)\label{X6}\\
&\qquad+\sum_{k\geq N+1}\sum_{j=0}^N(-1)^{j}{N+s+j\choose j} {N+s+k\choose k-j}J_{-s-k}(u)J_{k+t}(v). \label{X7}
\end{align}
In the expression (\ref{X6}), for $1\leq k\leq N$, we have
\begin{align*}
\sum_{j=0}^{k}&(-1)^{j}{N+s+j\choose j} {N+s+k\choose k-j}J_{-s-k}(u)J_{k+t}(v)\\
&=\sum_{j=0}^{k}
(-1)^{j}\dfrac{(N+s+j)!}{j!(N+s)!} \dfrac{(N+s+k)!}{(k-j)!(N+s+j)!}J_{-s-k}(u)J_{k+t}(v)\\
&=\sum_{j=0}^{k}
(-1)^{j}\dfrac{(N+s+k)!}{(N+s)!k!} \dfrac{k!}{(k-j)!j!}J_{-s-k}(u)J_{k+t}(v)\\
&={N+s+k\choose k}\sum_{j=0}^{k}(-1)^{j}{k\choose j}J_{-s-k}(u)J_{k+t}(v)\\
&=0,
\end{align*}
and for $k=0$, we will have $j=0$, so only one term will left in (\ref{X6}), which is $J_{-s}(u)J_t(b)$.
\end{proof}

\begin{cor}\label{keycoro}
In the universal enveloping algebra $U(V)$, for any integers $s, t$ and $N$ satisfying $N+s\geq 0$, we have the identity
\begin{align*}
&J_{-s}(u)J_t(v)\\
&=\sum_{j=0}^{N}\sum_{i\geq 0}(-1)^i{N+\Delta_u\choose i}{-N-s-1\choose j}J_{t-s}(u_{-N-s-i-j-1}v)\\
&\qquad-\sum_{k\geq N+1}\sum_{j=0}^N(-1)^{j}{N+s+j\choose j}{N+s-k\choose k-j}J_{-k-s}(u)J_{k+t}(v)\\
&\qquad+\sum_{j=0}^N\sum_{i\geq 0}(-1)^{N+s+1}{N+s+j\choose j}{N+s+j+i\choose i}J_{t-N-s-1-i}(v)J_{N+1+i}(u).
\end{align*}
\end{cor}
\begin{proof}
In the universal enveloping algebra, we have
\begin{align*}
\sum_{j=0}^{N}&{-N-s-1\choose j}\,^{(2)}J_{N+1, t+j, -N-s-1-j}^{u, v}\\
&=\sum_{j=0}^{N}{-N-s-1\choose j}\,^{2}J_{N+1+\Delta_u, t+j+\Delta_v, -N-s-1-j}^{u, v}\\
&= \sum_{j=0}^{N}{-N-s-1\choose j}\, ^{1}J_{N+1+\Delta_u, t+j+\Delta_v, -N-s-1-j}^{u, v}\\
&=\sum_{j=0}^{N}\sum_{i\geq 0}\,(-1)^i{-N-s-1\choose j}{N+\Delta_u\choose i}J_{t-s}(u_{-N-s-i-j-1}v).
\end{align*}
Then use Lemma \ref{Key Lemma}, we can get the desired identity.
\end{proof}

\bigskip

\bibliographystyle{amsplain}

\end{document}